\documentclass[10pt,reqno]{amsart}

\usepackage{amssymb}
\usepackage{xypic}
\xyoption{all}
\def\umono{\ar@{_{(}->}[u]}
\def\uumono{\ar@{_{(}->}[uu]}

\def\lmono{\ar@{_{(}->}[l]}
\def\llmono{\ar@{_{(}->}[ll]}

\newcommand{\Z}{{\mathbb Z}}

\newcommand{\Q}{{\mathbb Q}}

\newcommand{\C}{{\mathbb C}}

\newcommand{\holim}{\operatornamewithlimits{holim}}

\newcommand{\aut}[2]{\operatorname{aut}_{#1}(#2)}
\newcommand{\Aut}[2]{\operatorname{Aut}_{#1}(#2)}
\newcommand{\Hom}[2]{\operatorname{Hom}(#1,#2)}

\newcommand{\map}[3]{\operatorname{map}(#2,#3)_{#1}}
\newcommand{\mapp}[3]{\operatorname{map}_*(#2,#3)_{#1}}

\renewcommand{\ker}{\operatorname{Ker}\nolimits}

\newcommand{\A}{\ifmmode{\mathcal{A}}\else${\mathcal{A}}$\fi}
\newcommand{\K}{\ifmmode{\mathcal{K}}\else${\mathcal{K}}$\fi}
\newcommand{\U}{\ifmmode{\mathcal{U}}\else${\mathcal{U}}$\fi}
\newcommand{\T}{\ifmmode{\mathcal{T}}\else${\mathcal{T}}$\fi}
\newcommand{\FF}{\ifmmode{\mathcal{F}}\else${\mathcal{F}}$\fi}
\newcommand{\LL}{\ifmmode{\mathcal{L}}\else${\mathcal{L}}$\fi}

\newcommand{\ncov}[2]{#1 \langle #2 \rangle}

\newcommand{\m}{morphism}

\newtheorem{theorem}{Theorem}[section]
\newtheorem{proposition}[theorem]{Proposition}
\newtheorem{corollary}[theorem]{Corollary}
\newtheorem{lemma}[theorem]{Lemma}
\theoremstyle{definition}
\newtheorem{definition}[theorem]{Definition}
\newtheorem{remark}[theorem]{Remark}

\newtheorem{construction}[theorem]{Construction}

\def\N{{\mathbb{N}}}

\title[Localization genus]{Localization genus}

\author{Jesper M.~M\o ller}
\address{Matematisk Institut\\
 Universitetsparken 5\\
 DK--2100 K\o benhavn}
\email{moller@math.ku.dk}
\urladdr{htpp://www.math.ku.dk/~moller}

\author{J\'er\^{o}me Scherer}
\address{\'Ecole Polytechnique f\'ed\'erale de Lausanne\\
 CH-1015 Lausanne, Switzerland}
\email{jerome.scherer@epfl.ch}
\urladdr{http://hessbellwald-lab.epfl.ch/Scherer}

\thanks{The first author is supported by the Danish National Research Foundation through
  the Centre for Symmetry and Deformation (DNRF92) and by Villum
  Fonden through the project Experimental Mathematics in Number
  Theory, Operator Algebras, and Topology. The second author is supported by FEDER/MEC grant
MTM2013-42293-P}


\begin{document}
\date{\today}
\maketitle

\begin{abstract}
Which spaces look like an $n$-sphere through the eyes of the
$n$-th Postnikov section functor and the $n$-connected cover
functor? The answer is what we call the Postnikov genus of the
$n$-sphere. We define in fact the notion of localization genus for
any homotopical localization functor in the sense of Bousfield and
Dror Farjoun. This includes exotic genus notions related for example
to Neisendorfer localization, or the classical Mislin genus, which 
corresponds to rationalization.
\end{abstract}


\maketitle


\section*{Introduction}
\label{sec intro}
Classically the genus of a nilpotent space $X$ of finite type, as
introduced by Mislin in \cite{mislin71}, consists of all homotopy
types of nilpotent spaces $Y$ of finite type such that the
localizations $Y_{(p)}$ and $X_{(p)}$ coincide at any prime $p$.
That is, spaces in the same genus as $X$ cannot be distinguished
from $X$ if one looks at them through the eyes of $p$-localization.
Another equivalent definition can be given in terms of
rationalization and $p$-completion, \cite{MR0402737}.

We introduce in this article the notion of \emph{localization
genus}. A localization functor $L$ in the category of spaces (or
simplicial sets), as introduced by Bousfield, \cite{MR57:17648},
Farjoun, \cite{Dror}, is a homotopy functor equipped with a natural
transformation $\eta$ from the identity which is idempotent up to
homotopy. The main point to study such functors is that it
subsumes the notions of localization at a prime or a set of primes (e. g.
rationalization) and $p$-completion, but also Postnikov sections,
Quillen's plus-construction, and other nullification or
periodization functors such as $P_{B\Z/p}$, which plays a central
role in the Sullivan conjecture, \cite{Miller}. We write $\bar L X$
for the homotopy fiber of the natural map $\eta_X: X \rightarrow
LX$. We define thus two genus sets associated to~$L$ for any simply connected
CW-space $X$ of finite type.

\begin{enumerate}
 \item The {\em extended $L$ genus set\/} for $X$ is the set
 $\bar G_L(X) = \{ Y \mid LY \simeq LX, \ \bar LY \simeq \bar LX \}$
 of homotopy types $Y$ of CW-spaces such that $LY=LX$ and $\bar LY =
 \bar LX$.

 \item The {\em $L$-genus set\/} for $X$ is the subset $G_L(X)$ of
 $\bar G_L(X)$ represented by CW-spaces of finite type.
\end{enumerate}

Our definition is motivated by the classical definition of the
(completion) genus set, \cite{mislin71}, \cite[Definition
3.2]{MR0402737}, and the extended (completion) genus set studied by
McGibbon in \cite{chuck94}. We show in fact in Proposition~\ref{prop:classicalG}
that when $L$ is rationalization, one gets back these classical notions. To illustrate
our point of view we go through the computation of the extended rationalization
genus $\bar G(S^n)$ of an odd sphere, Theorem~\ref{thm:extendedgenus}, and we 
characterize in Corollary~\ref{cor:hiltongenus} those elements in 
$\bar G(S^n)$ corresponding to elements in the extended genus of the \emph{abelian group}
of integers, as studied by Hilton in \cite{hilton88pseudo}.

To tackle technically harder problems we rely on Dywer and Kan's classifiyng space
for towers of fibrations, \cite{dwks:towers}, a tool which has proven to be handy in
similar situations, \cite{MR2669687}. This allows us in particular to do explicit computations
of Postnikov genus sets for spheres and complex projective spaces.

\medskip

\noindent
{\bf Theorem~\ref{thm:classification}.}
\emph{The extended Postnikov genus set $\overline G_{[n]}(S^n)$ of homotopy types
of spaces $Y$ such that $Y[n] \simeq K(\Z, n)$ and $Y\langle n
\rangle \simeq S^n \langle n \rangle$ is uncountable, in bijection
with $\prod_p \N_+$, where the product is taken over all primes.}

\medskip

We also present in Section~\ref{sec:joe} a computation related to Neisendorfer's functor, \cite{MR96a:55019}, 
and the Sullivan conjecture. The localization genus
computations show combined features of the space one focuses on and the chosen localization functor. The
notion of genus quantifies in which sense it is (not) sufficient to consider a given space locally, through the
eyes of a localization functor $L$ and the associated fiber~$\bar L$.

\medskip

\noindent
{\bf Acknowledgements.} This project started during a visit of the first author at the Universitat Aut\`onoma de Barcelona in 2008 
and was rebooted during a visit of the second author at the Centre for Symmetry and Deformation in Copenhagen seven years later.

\section{Genus and extended genus}
\label{sec:definition}
Let $L$ be a homotopical localization functor, i.e. a coaugmented and idempotent homotopy functor in the category of spaces. It 
is sometimes more convenient to work in the Quillen equivalent category of simplicial sets, in particular when one needs
models for mapping spaces. We will clearly say so when we do so. In practice localization functors arise as follows. To any map $f$
one associates a functor $L_f$ which inverts $f$ in a universal way, \cite{Dror} and \cite{MR57:17648}. Clever choices for the
map yield homological localization, localization at a set of primes such as rationalization, Quillen's plus construction, Postnikov sections, etc.

The homotopy fiber of the coaugmentation $X \rightarrow LX$ is $\overline{L} X$.

\begin{definition}
\label{def:Lgenus}
{\rm Let $L$ be a localization functor and $X$ a simply connected
 CW-complex of finite type.
 \begin{itemize}
 \item The {\em extended-$L$ genus set\/} for $X$ is the set
\begin{equation*}
 \overline{G}_L(X) = \{ Y \mid LY \simeq LX, \ \overline{L}Y \simeq \overline{L}X \}
\end{equation*}
of homotopy types $Y$ of CW-spaces such that $LY=LX$ and $\overline{L}Y =
\overline{L}X$.
\item The {\em $L$-genus set\/} for $X$ is the subset $G_L(X)$ of
 $\overline{G}_L(X)$ represented by CW-complexes of finite type.
 \end{itemize}}
\end{definition}

The reason for this ``generic'' terminology comes from the relationship with the classical notion of genus.
Recall that Mislin's definition, \cite{mislin71}, is given in terms of localization at primes: Two spaces $X$ and $Y$ 
belong to the same genus set if their localizations $X_{(p)}$ and $Y_{(p)}$ are homotopy equivalent at every prime $p$. 
This is a stronger requirement than merely asking for equivalent $p$-completions $X_p^\wedge$ and $Y_p^\wedge$, 
for any prime $p$, and equivalent rationalizations $X_{0}$ and $Y_{0}$,
as shown for example by Belfi and Wilkerson in \cite{MR0418086}. We will focus on the completion genus set
as in \cite[Definition~3.5]{MR0402737} and the extended completion genus set, \cite{chuck94}. We denote by $X^\wedge$
the product of all $p$-completions.

\begin{definition}
\label{def:oldgenus}
 Let $X$ be a simply connected CW-complex of finite type.
 \begin{itemize}
 \item The {\em extended genus set\/} of $X$ is the set $\overline{G}(X)$ of
   homotopy types of CW-complexes $Y$ such that $Y^\wedge = X^\wedge$
   and $Y_0 = X_0$.
 \item The {\em genus set\/} of $X$ is the subset $G(X)$ of $\overline{G}(X)$
   represented by CW-complexes $Y$ of finite type.
 \end{itemize}
\end{definition}

We show now that the classical completion genus coincides with our localization
genus, when the chosen localization functor $L$ is rationalization. Since we restrict
our attention to simply connected spaces here, one can choose the map $f$ to be the
wedge of degree $p$ maps on the $2$-sphere, taken over all primes $p$. Then the
Bousfield localization functor $L_f$ is rationalization on simply connected spaces.
When $LX=X_0$ is rationalization, we write $X_\tau$ for $\overline{L}X$, the fibre of the
rationalization map $X \to X_0$ and call it the \emph{torsion space} of $X$.

\begin{proposition}
\label{prop:classicalG}
Let $X$ be a simply connected CW-complex of finite type.  The
extended rationali\-zation-genus set
 \begin{equation*}
   \overline{G}_0(X) =  \{ Y \mid Y_0 \simeq X_0, \
   Y_\tau \simeq X_\tau \} =
   \{ Y \mid Y_0 \simeq X_0, \
   Y^\wedge \simeq X^\wedge \}
 \end{equation*}
 is the extended genus set $\overline{G}(X)$ and the rationalization-genus set $G_0(X)$
 coincides with the classical genus set $G(X)$.
\end{proposition}

\begin{proof}
To see this, note that $X_\tau \simeq Y_\tau$ if and only if  $X^\wedge \simeq
Y^\wedge$. Indeed, if we complete the fibration $X_\tau \to X \to
X_0$ (by the nilpotent fibration Lemma \cite[II.4.8]{MR51:1825}) 
we see that $(X_\tau)^\wedge = X^\wedge$, and Sullivan's
arithmetic square, \cite[VI.8.1]{MR51:1825}, shows that $X_\tau$ is (also) the fibre of
$X^\wedge \to (X^\wedge)_0$.
\end{proof}

Let us finally remark that all spaces
in $G(X)$ are finite complexes when $X$ is a finite complex. This
comes from the fact that, when $X$ is of finite type, the integral
homology groups of any space in the genus set of $X$ are those of
$X$. There is thus a Moore-Postnikov decomposition of such a space
as successive homotopy cofibers of maps between (finite) Moore
spaces, see for example \cite[Chapter~8]{MR0198466}.

\section{The extended rationalization genus of an odd sphere}
\label{sec:ratodd}
In this section we turn our attention to a concrete example and propose an 
explicit computation of the rationalization-genus for odd spheres.
Let $n$ be an odd natural number. The extended rationalization genus set of the
odd-dimensional sphere $S^n$
is according to \cite[Theorem~3]{chuck94} an uncountable set. We
offer an explicit description of this extended genus set and identify the
elements known as pseudo-spheres. For this we will need some elementary abelian
group theory.

Any torsion free abelian group $A$ of rank one can be seen, up to
isomorphism, as a subgroup of $\Q$ containing $\Z$. For each prime
$p$, let $k_p(A) = \max \{ r \geq 0 \mid 1 \in p^rA \}$ denote the
\emph{height} of $1$ at $p$. The {\em height sequence\/} of $A$ is the
sequence $\Pi(A) = (k_p(A))_p$ of non-negative (or infinite)
integers. Two sequences $(k_p)$ and $(m_p)$ are similar if the sum of
the differences $|k_p -m_p|$ is finite. This means that the sequences
differ in only a finite number of primes, and have $\infty$ in the same
coordinate. A {\em type\/} is a
similarity class of sequences. As explained in \cite{MR41:333},
or \cite[Theorem 10.47]{rotman},
isomorphism types of torsion free abelian groups of rank one are in
bijection with types.

We start now with the fibration $S^n_\tau \rightarrow S^n
\rightarrow (S^n)_{0}$ and will use that, since $n$ is odd,
$(S^n)_{0} \simeq K(\Q, n) \simeq M(\Q, n)$ and $S^n_\tau \simeq
M(\Q/\Z, n-1)$. For any space $Y$ in the extended rationalization
genus of $S^n$ we have a fibration sequence $M(\Q/\Z, n-1)
\rightarrow Y \rightarrow K(\Q, n)$. The homotopy long exact
sequence yields an exact sequence
\[
0 \rightarrow \pi_n Y \rightarrow \Q \xrightarrow{\partial} \Q/\Z
\rightarrow \pi_{n-1} Y \rightarrow 0 \, .
\]
The idea is that the connecting homomorphism $\partial$ determines
the homotopy type of $Y$.

\begin{lemma}
\label{lem:rankone}
There is a bijection between isomorphism types of torsion free
abelian groups of rank one and the double coset $\Q^\times
\backslash \Hom{\Q} {\Q/\Z} / \hat \Z$, where the units in $\Q$ act
by pre-composition and the automorphisms $\hat \Z$ of $\Q/\Z$ by
post-composition.
\end{lemma}

\begin{proof}
We define two maps. The first $\alpha: \Hom{\Q} {\Q/\Z} \rightarrow
\{ A \, | \, A \ \hbox{\rm torsion free abelian of rank one} \}$
sends a homomorphism $\partial$ to its kernel. The second one we call $\beta$. 
Given a subgroup $A$ of $\Q$, let $J$ be the subset of all primes
consisting of those primes $p$ for which $n_p \neq \infty$. Then the
quotient $\Q/A$ is isomorphic to $\oplus_{p\in J} \Z_{p^\infty}$ and
$\beta$ sends $A$ to the composite
$$
\Q \rightarrow \Q/A \xrightarrow{\cong} \oplus_{p\in J}
\Z_{p^\infty} \hookrightarrow \Q/\Z
$$
Clearly $\alpha \circ \beta$ sends a torsion free abelian group of
rank one $A$ to a subgroup of $\Q$ which is isomorphic to~$A$.
Moreover, given a homomorphism $\partial: \Q \rightarrow \Q/\Z$,
the image $\beta(\ker \partial)$ coincides with $\partial$ up to an
isomorphism of $\Q$ (corresponding to the choice of an inclusion
$\alpha(\partial) \subset \Q$) and an isomorphism of $\Q/\Z$
(corresponding to the choice of an isomorphism $\Q/\ker \partial
\cong \oplus_{p\in J} \Z_{p^\infty}$). This proves the lemma.
\end{proof}

We proceed now with the construction for any homomorphism $\partial:
\Q \rightarrow \Q/\Z$ of a space $Y(\partial)$ in the extended genus
$\overline{G}(S^n)$
realizing this homomorphism as connecting map in the homotopy long
exact sequence of the fibration $M(\Q/\Z, n-1) \rightarrow
Y(\partial) \rightarrow K(\Q, n)$. There is a bijection $[M(\Q,
n-1), M(\Q/\Z, n-1)] \cong \Hom {\Q} {\Q/\Z}$ and so there exists up
to homotopy a unique map $\Delta: M(\Q, n-1) \rightarrow M(\Q/\Z,
n-1)$ such that $\pi_{n-1} (\Delta) =
\partial$. We define $Y(\partial)$ to be the homotopy cofiber of $\Delta$. We
have therefore a cofibration sequence
$$
M(\Q/\Z, n-1) \rightarrow Y(\partial) \rightarrow M(\Q, n)
$$
which is seen to be a fibration sequence as well, for example by an
elementary Serre spectral sequence argument (a complete characterization of
such sequences has been obtained by Alonso, \cite{MR0977927},
see also Wojtkowiak's \cite{MR0559054}).

\begin{theorem}
\label{thm:extendedgenus}
The extended genus set $\overline{G}(S^n)$ is in bijection with the set of
isomorphism classes of torsion free abelian groups of rank one.
\end{theorem}

\begin{proof}
Given a torsion free abelian group $A$ of rank one we get a
homomorphism $\partial: \Q \rightarrow \Q/\Z$ from
Lemma~\ref{lem:rankone} and construct as above the space
$Y(\partial)$. It realizes $\partial$ as connecting homomorphism in
the homotopy long exact sequence, and its kernel is a group
isomorphic to~$A$. To show that we have indeed a bijection we must
prove that the connecting homomorphism determines the homotopy type
of $Y \in \overline{G}(S^n)$. Let $\partial$ be the connecting
homomorphism for such a space $Y$ and let us compare $Y$ and
$Y(\partial)$. We have a map $i: M(\Q/\Z, n-1) \rightarrow Y$ by
definition of the extended genus set and consider the composite $i
\circ \Delta: M(\Q, n-1) \rightarrow Y$, where $\Delta$ is, as above, the
the unique map,  up to homotopy, realizing $\partial$. It coincides thus, up to
homotopy, with the composite of the Postnikov section $M(\Q, n-1)
\rightarrow K(\Q, n-1)$ and the homotopy fiber inclusion of $K(\Q,
n-1)$ in the total space of the map $M(\Q/\Z, n-1) \rightarrow Y$.
This shows that the composite $i \circ \Delta$ is homotopically
trivial. Therefore the map $i$ factors through the homotopy cofiber
of $\Delta$, i.e. $Y(\partial)$.

We have thus constructed a map $Y(\partial) \rightarrow Y$ which
induces equivalences on rationalizations and torsion spaces. It is
hence an equivalence as well.
\end{proof}

In the final part of the section we restrict our attention to the
$(n-1)$-connected members of the extended genus set of an odd
sphere. We begin with a review of Hilton's investigations of the
extended genus set of $\Z$ and groups of pseudo-integers
\cite{hilton88pseudo,hilton88}.

\begin{definition}\cite{hilton88pseudo}
\label{defn:extendedGZ}
A subgroup of the full rational group $\Q$ is a {\em group of
pseudo-integers\/} if it contains $\Z$ but not $\Z[1/p]$ for any
prime $p$.
\end{definition}

Since a group of pseudo-integers is a torsion free abelian group of
rank one, in the terminology introduced at the beginning of the
section, it is characterized by its type, which consists in only
finite integers $k_p$. This subset of torsion free abelian group of rank one 
has been studied by Hilton in \cite[Theorem 2.3, 2.4]{hilton88pseudo}.

According to \cite[Corollary 2.5]{hilton88pseudo}, the extended genus
set $\overline{G}(\Z)$ of $\Z$, consisting of isomorphism classes of (not
necessarily finitely generated) abelian groups $H$ such that $H
\otimes \Z_{(p)} \cong \Z_{(p)}$ for all primes $p$, is the set of
iso\m\ classes of pseudo-integers.  In other words, $\bar G(\Z)$ is
the set of iso\m\ classes of torsion free abelian groups of rank $1$
of $\infty$-free types \cite[\S 42]{MR41:333}.

\begin{corollary}
\label{cor:hiltongenus}
The set of $(n-1)$-connected spaces in $\overline{G}(S^n)$ is in bijection
with $\overline{G}(\Z)$. The correspondence is given by $Y \mapsto \pi_n(Y)$.
\end{corollary}

\begin{proof}
The spaces in the extended genus of $S^n$ which are
$(n-1)$-connected are characterized by the fact that the connecting
homomorphism $\partial$ is surjective. In other words its kernel is
a group of pseudo-integers. We conclude by
Theorem~\ref{thm:extendedgenus}.
\end{proof}

\section{A formula to compute localization genera}
\label{sec genus}
In the previous section we have been able to establish a complete
and explicit list of all homotopy types in the extended
rationalization genus of an odd sphere. In general, for arbitrary
spaces and arbitrary localization functor, this is not to be
expected. Following the approach of Dwyer, Kan, and Smith in
\cite{dwks:towers} to classify towers of fibrations, we propose in
this section a formula which we use later on to perform
computations of ``Postnikov'' and ``Neisendorfer'' genus. 
We start with the necessary background from~\cite{dwks:towers}.

Let $G$ be a space and consider the functor $\Phi$ which sends an
object of $Spaces \downarrow B \aut{} {G}$, i.e. a map $t: X
\rightarrow B \aut{} {G}$, to the twisted product $X \times_t G$.
Dwyer, Kan, and Smith describe a right adjoint $\Psi$ in
\cite[Section~4]{dwks:towers}. They find first a model for $\aut{}
{G}$ which is a (simplicial) group and thus acts on the left on
$\map{} {G} Z$ for any space $Z$. This induces a map $r: B \aut{}
{G} \rightarrow B\aut{}{\map{} {G} Z}$. The functor $\Psi$ sends
then $Z$ to the projection map from the twisted product $B \aut{}
{G} \times_r \map{} {G} Z \rightarrow B \aut{} {G}$. This allows
right away to construct a classifying space for towers, in our case
they will be of length $2$.

\begin{theorem}\label{thm:dks}{\rm Dwyer, Kan, Smith,
\cite{dwks:towers}.}
The classifying space for towers of the form $Z \xrightarrow{q} Y
\xrightarrow{p} X$, where the homotopy fiber of $p$ is $G$ and that
of $q$ is $H$, is $B \aut{} {G} \times_r \map{} {G} {B \aut{} {H}}$.
\end{theorem}

\begin{remark}
\label{rem:simplicial}
Working with the Dwyer-Kan-Smith model means that we deal with simplicial sets. The comparison
with spaces is via the singular complex and geometric realization. Since the realization of a simplicial
set is a CW-complex, all spaces we construct here have obviously the homotopy type of a CW-complex.
\end{remark}

There is an interesting consequence of Theorem~\ref{thm:dks}, which
can be compared with our former result about the classifying space
of the monoid of self-equivalences of a two-stage Postnikov piece,
\cite[Theorem~5.3]{MR2669687}.

\begin{theorem}
\label{thm:fakes}
Let $L$ be a localization functor and $X$ be a space such that $L
\overline{L}X$ is contractible. There is then a bijection between $G_L(X)$
and the set
 \begin{equation*}
   [LX,B\aut{}{\bar LX}]/\Aut{}{LX}
 \end{equation*}
of orbits for the action of the group $\Aut{}{LX}$ on the set
$[LX,B\aut{}{\bar LX}]$.
\end{theorem}

\begin{proof}
Equivalence classes of towers of fibrations $Z \to Y \to *$ with
successive fibres $LX$ and $\bar L X$ are classified by the set of
path components of  $B=B\aut{}{LX} \times_r \map{}{LX}{B\aut{}{\bar
LX}}$. It is clear that any such space $Z$ fits in a fibration $\bar
LX \rightarrow Z \rightarrow LX$, but we must show that the map $Z
\rightarrow LX$ coincides with the localization map for~$Z$.
Fibrewise localization, \cite[Theorem 1.F.1]{Dror}, yields a natural
transformation between fibration sequences
\[
\xymatrix{
\overline{L} X \rto \dto & Z \rto \dto & LX \ar@{=}[d]\\
L\overline{L} X \rto & \overline{Z} \rto & LX
}\]
By assumption $L\overline{L} X$ is contractible and since
by construction the map $Z \rightarrow \overline{Z}$ is an $L$-equivalence
it must coincide with the localization map $Z \rightarrow LX$.
\end{proof}

The assumption that $L \overline{L}X$ is contractible is 
restrictive if we would require this for all spaces~$X$. This would
amount to imposing that the localization functor $L$ is a so-called
\emph{nullification} functor, i.e. a homotopical localization functor
associated to a map of the form $A \rightarrow *$ such as a Postnikov
section -- when $A$ is a sphere, \cite[1.A.6]{Dror} -- or Quillen's plus 
construction. We impose this condition however on a single space,
and this happens sometimes for localization functors that are not
nullifications. When $X$ is a simply connected
space of finite type and $L$ is rationalization, $p$-completion, or
completion, then $L \overline{L} X \simeq *$. 
These are the main examples of interest in this note.

According to Theorem~\ref{thm:fakes}, one can often compute a
localization genus set via a Dwyer-Kan-Smith type formula.

\begin{corollary}
\label{cor:classicalDKS}
The classical extended genus set of a simply connected space $X$ of finite
type is given by $\overline{G}(X) = [X_0, B\aut{}{X_\tau}]/\Aut{}{X_0}$.
\end{corollary}

\section{Neisendorfer genus and Postnikov genus}
\label{sec:joe}
Let $P$ be the nullification functor \cite[1.A.4]{Dror} with respect
to the wedge $\bigvee\! B\Z/p$ taken over all primes $p$. The unexpected
effect of $P$ on highly connected covers of finite spaces has been first
studied by Neisendorfer, \cite{MR96a:55019}. It relies on Miller's solution
to the Sullivan conjecture, \cite{Miller}. 
The name \emph{Neisendorfer's functor} is commonly used for $P_{B\mathbb Z/p}$ 
followed by completion at the prime $p$. This composition of two localization functors
however is not itself a localization functor because it fails to be idempotent in general.
Therefore we use here the name Neisendorfer's functor for localization with respect
to the maps $c\colon \bigvee\! B\Z/p \rightarrow *$ and a wedge of universal mod $p$ 
homology equivalences~$h$. Thus $L_h$ coincides with profinite completion on nilpotent spaces and
$\hat P X= L_{c \vee h} X$ often agrees with $(PX)^\wedge$, e.g. when the latter space is already $B\Z/p$-null.  

\begin{theorem}\cite[Theorem 4.1]{MR96a:55019}
\label{thm:joe}
Let $X$ be a simply connected finite complex with $\pi_2(X)$
finite and let $n \geq 1$ be a natural number. Then
\begin{equation*}
 P(\ncov Xn) = \holim \left( X \to X_0 \leftarrow \ncov Xn_0 \right)
\end{equation*}
so that $P(\ncov Xn)_\tau \simeq X_\tau$ and $P(\ncov Xn)^\wedge \simeq
X^\wedge$.
\end{theorem}

As a direct consequence we compute the Neisendorfer genus of a finite complex. In a given localization genus set we fix
the localization and the fiber of the localization map. Thus, even if all highly connected covers of a finite complex have the
same Neisendorfer localization as the original complex, they do not belong to the (extended) genus set since the fibers
of the localization maps fail to agree.

\begin{corollary}
\label{cor:joegenus}
Let $X$ be a simply connected finite complex with $\pi_2 X$ finite.
Then $\overline{G}_P(X) = G_P(X)  =  \{ X \}$.
\end{corollary}

\begin{proof}
Applying Theorem~\ref{thm:joe} to $X$ with $n=1$ we see that $PX \simeq X$. 
\end{proof}

For us the following consequences will be important. In particular
point (3) will help compute the monoids of self-equivalences which
appear in the formula from Corollary~\ref{cor:classicalDKS}.

\begin{corollary}
\label{cor:joe}
Let $X$ be a simply connected finite complex with $\pi_2(X)$ finite.
Suppose that $\pi_{>N}(X) \otimes \Q = 0$ for some integer $N$. Then
 \begin{enumerate}
 \item $P(\ncov XN) \simeq X_\tau$
 \item \label{cor:joe2} \cite{MR1444714} There are weak
   equivalences
\begin{equation*}
 \aut *{\ncov XN} \simeq \aut *{X_\tau} \simeq \aut *{X^\wedge} \simeq  \aut *{\ncov {X^\wedge}N}
\end{equation*}
of topological monoids of pointed self homotopy equivalences
\item The obvious map $G(X) \to G(X[N])$ is bijective
 \end{enumerate}
\end{corollary}

\begin{proof}
Since $\ncov XN_0$ is contractible we deduce from
Theorem~\ref{thm:joe} that $P(\ncov XN) = X_\tau$, which proves~(1).
Hence starting with $\ncov XN$ we see that we get first $X_\tau$ by
applying $P$, then $X^\wedge$ by applying completion, and finally
$\ncov XN$ by taking the $N$-connected cover. As this last functor
is a pointed one (it can be seen for example as cellularization with respect to the
sphere $S^{N+1}$, \cite[Example 2.D.6]{Dror}), we obtain a chain of weak homotopy equivalences
$\aut *{\ncov XN} \simeq \aut *{X_\tau} \simeq \aut *{X^\wedge}
\simeq  \aut *{\ncov {X^\wedge}N}$. This shows (2).

Wilkerson's double coset formula for the genus set \cite[Theorem~3.8]{MR0402737} 
exhibits $G(X)$ as double coset of the so-called $\pi_*$-continuous self-equivalences of $(X^\wedge)_0$
under the left action of $\aut{}{X^\wedge}$ and the right action of $\aut{}{X_0}$.
Clearly $X$ and  the Postnikov section $X[N]$ are rationally identical by assumption. Also the groups of self-equivalences 
are the same for $X$ and its completion by part (2) of this corollary (which is for the pointed version but the spaces here 
are simply connected). Thus all ingredients in Wilkerson's formula are identical for $X$ and $X[N]$, which
shows that the map $G(X) \to G(X[N])$ is bijective.
\end{proof}

We note that $G(X[N])$ can be computed from Zabrodsky's exact
sequence when $X_0$ is an $H$-space \cite{alex74H,alex74},
\cite[Theorem 4]{chuck94}.

We turn now to our first computation of Postnikov genus. If $f:S^{N+1} \rightarrow *$ is the constant map,
then  $L_f$ is a functorial $N$-th Postnikov section, $L_f X \simeq X[N]$, and $\overline{L}_f$ is a functorial $N$-connected cover,
$\bar L_f X \simeq X \langle N \rangle$.

\begin{definition}
\label{def:Postnikovgenus}
The \emph{$N$-th Postnikov section genus} of a space $X$ is the set $G_L(X)$ when $L$ is localization with respect
to $f:S^{N+1} \rightarrow *$. We write $G_{[N]}(X)$ for this set. 
\end{definition}

Hence a space $Y$ belongs to the extended genus $\overline{G}_{[N]}(X)$ if  its $N$-th Postnikov section $Y[N]$ 
coincides with $X[N]$ and its $N$-connected cover $Y\langle N \rangle$ coincides with  $X\langle N \rangle$.

\begin{theorem}
\label{thm:Postnikovgenus}
Let $X$ be a simply connected finite complex with $\pi_2(X)$ finite and
suppose that $\pi_{>N}(X) \otimes \Q = 0$ for some integer $N$. Then the 
only finite CW-complex in $G_{[N]}(X)$ is $X$
\end{theorem}

\begin{proof}
If $Y$ belongs to  $G_{[N]}(X)$, then $\ncov YN \simeq \ncov XN$ and $Y[N] \simeq
X[N]$. Thus, if $Y$ is finite, $Y^\wedge \simeq P(\ncov YN)^\wedge
\simeq P(\ncov XN)^\wedge \simeq X^\wedge$ by Theorem~\ref{thm:joe}
and $Y_0 \simeq Y[N]_0 \simeq  X[N]_0 \simeq X_0$ showing that $Y
\in G(X)$. But $X$ and $Y$ have the same image under the injective
map $G(X) \to G(X[N])$ (by the previous theorem) so $X\simeq Y$.
\end{proof}

\begin{remark}
\label{rem:infiniteCW}
Point (4) above tells us that there are very few finite CW-complexes
in a Postnikov genus set. However there are many infinite
CW-complexes of finite type in the extended genus set. For example when $X = S^3$ and $L$ is
the third Postnikov section, the space $K(\Z, 3) \times S^3\langle 3
\rangle$ is obviously in the $L$-genus of $S^3$. We will come back
to this kind of example with a detailed computation in the next
section.
\end{remark}

\begin{remark}
\label{rem:split}
It is not always true that there is a single finite complex in the
Postnikov genus of a finite complex $X$. Let us consider for example
the space $S^2 \times S^5$ and the functor $L$ is chosen to be the
second Postnikov section. Then $(S^2 \times S^5)[2] \simeq K(\Z, 2)$
and $(S^2 \times S^5)\langle 2 \rangle \simeq S^3 \times S^5$. It is
easy to see that the space $\C P^2 \times S^3$ also belongs to
$G_{[2]}(S^2 \times S^5)$. Of course the condition of the corollary are
not fulfilled since neither $\pi_2 S^2$, nor $\pi_3 S^2$, are torsion.

It would be interesting to construct similar examples with higher Postnikov sections
and at least $2$-connected spaces, so the $\pi_2$ assumption is trivially fulfilled.
\end{remark}

\section{Self-equivalences of connected covers of a sphere}
\label{sec selfequivalences}
Our next goal will be to determine the $n$-th Postnikov genus of an odd sphere $S^n$ with $n \geq 3$.
This will be done by using Theorem~\ref{thm:fakes}, which involves the
computation of the space of self-equivalences of the $n$-th connected cover $S^n \langle n \rangle$.
This section prepares the terrain for the genus computation in the next section and
focuses on handy properties of $\aut{}{S^n \langle n \rangle^\wedge_p}$.

We write $X^\wedge_p$ for the $p$-completion of $X$. Since $S^n
\langle n \rangle$ is a torsion space, it is weakly equivalent to
the product of its $p$-completions $S^n \langle n \rangle^\wedge_p$.
Now
\[
\map{}{S^n \langle n \rangle}{S^n \langle n \rangle} \simeq \prod_p
\map{}{S^n \langle n \rangle}{S^n \langle n \rangle^\wedge_p} \simeq
\prod_p \map{}{S^n \langle n \rangle^\wedge_p \times \prod_{q \neq
p} S^n \langle n \rangle^\wedge_q}{S^n \langle n \rangle^\wedge_p}
\]
But since $S^n \langle n \rangle^\wedge_p$ is $p$-complete and
$(\prod_{q \neq p} S^n \langle n \rangle^\wedge_q)^\wedge_p$ is
contractible, we see that this mapping space is weakly equivalent to
$\prod_p \map{}{S^n \langle n \rangle^\wedge_p}{S^n \langle n
\rangle^\wedge_p}$. Therefore the subspace of self-equivalences also
splits as a product
\[
\aut{}{S^n \langle n \rangle} \simeq \prod_p \aut{}{S^n \langle n
\rangle^\wedge_p}.
\]
Because of the formulas in Theorem~\ref{thm:fakes}
we wish to understand certain mapping spaces into $B \aut{}{S^n
\langle n \rangle^\wedge_p}$ and start with a few elementary lemmas.

\begin{lemma}
\label{lem:bzplocal}
The space $B\aut{*}{(S^n)^\wedge_p} \simeq B\aut{*}{S^n \langle n \rangle^\wedge_p}$ is $\Sigma
B\Z/p^k$-local for any $k \geq 1$.
\end{lemma}

\begin{proof}
From Corollary~\ref{cor:joe}.\eqref{cor:joe2} we know that the spaces of
pointed self-homotopy equivalences of the $p$-completed sphere and its
$n$-connected cover coincide.
By adjunction the pointed mapping space $\mapp{}{\Sigma B\Z/p}
{B\aut{*}{(S^n)^\wedge_p}}$ is equivalent to
$\mapp{}{B\Z/p} {\aut{*}{(S^n)^\wedge_p}}$. But $\aut *{S^n
\langle n \rangle^\wedge_p}$ consists of certain components of
$\mapp{}{(S^n)^\wedge_p} {(S^n)^\wedge_p} \simeq \mapp{}{S^n}
{(S^n)^\wedge_p} = \Omega^n (S^n)^\wedge_p$, and all components of this
iterated loop space have the same homotopy type, and they are
$B\Z/p$-local by Miller's Theorem~\cite{Miller}.
\end{proof}

The following is much easier to prove for the connected cover than for the $p$-completed sphere itself.

\begin{lemma}
\label{lem:1/plocal}
The space $B\aut{*}{(S^n)^\wedge_p} \simeq B\aut{*}{S^n \langle n \rangle^\wedge_p}$ is 
$K(\Z[1/p], 2)$-local.
\end{lemma}

\begin{proof}
Since $K(\Z[1/p], 2)$ is $S^2[1/p]$-cellular it is sufficient to
prove that the classifying space is $S^2[1/p]$-local. We use the
usual telescopic model for this localized sphere, i.e. the homotopy
colimit of the diagram $S^2 \xrightarrow{p} S^2 \xrightarrow{p}
\dots$. Hence we have weak equivalences
$$
\mapp{}{S^2[1/p]}{B\aut{*}{S^n \langle n \rangle^\wedge_p}} \simeq
\holim \mapp{}{S^2}{B\aut{*}{X_p}} = \holim \Omega \aut{*}{S^n
\langle n \rangle^\wedge_p}.
$$
The homotopy groups of this inverse limit are all trivial since the
towers we consider here consist of finite $p$-groups and
multiplication by~$p$ (in particular all $\hbox{\rm lim}^1$ terms vanish).
\end{proof}

Even though we are not so sure whether $B\aut{*}{S^n \langle n
\rangle^\wedge_p}$ is $p$-complete, the previous lemmas allow us to
understand maps out of $K(\Z, n)$.

\begin{proposition}
\label{prop:Klocal}
For any $m \geq 3$ the space $B\aut{*}{(S^n)^\wedge_p} \simeq B\aut{*}{S^n \langle n \rangle^\wedge_p}$ 
is $K(\Z, m)$-local. In particular it is $K(\Z, n)$-local.
\end{proposition}

\begin{proof}
The space $K(\Z_{p^\infty}, m-2)$ is $\bigvee B\Z/p^k$-cellular, being a telescope of 
$K(\mathbb Z/p^k, m-2)$'s which are cellular. Hence 
$K(\Z_{p^\infty}, m-1)$ is $\bigvee \Sigma B\Z/p^k$-cellular
by the commutation rule of cellularization
with respect to loop spaces, \cite[Theorem~3.A.2]{Dror}.
Lemma~\ref{lem:bzplocal} implies then that
$\mapp{}{K(\Z_{p^\infty}, m-1)} {B\aut{*}{S^n \langle n \rangle^\wedge_p}}$
is contractible. Zabrodsky's Lemma \cite[Proposition~3.4]{MR97i:55028}
produces then an equivalence between $\mapp{}{K(\Z, m)}{B\aut{*}{S^n \langle n
\rangle^\wedge_p}}$ and $\mapp{}{K(\Z[1/p], m)}{B\aut{*}{S^n \langle n
\rangle^\wedge_p}}$, which is contractible by Lemma~\ref{lem:1/plocal}.
\end{proof}

We point out that the argument with a wedge is only necessary for $m=3$. For any larger value of $m$
we could have gone through the same proof with $\Sigma B \mathbb Z/p$.

\section{The extended Postnikov genus of an odd sphere}
\label{sec sphere}
We come now to our most sophisticated computation. We wish to
determine the extended Postnikov genus $\overline G_{[n]}(S^n)$ when $n$ is
odd and $X[n]$ is the $n$-th Postnikov section of $X$. In other words, we wish
to understand how many spaces $Y$ are extensions of $K(\Z, n)$ by
$S^n\langle n \rangle$, i.e. how many spaces look like a sphere
$S^n$ through the eyes of the $n$-th Postnikov section \emph{and}
the $n$-th connected cover functors. We write $S^n_p$ for
the fiberwise $p$-completion of $S^n$ sitting in the fibration $S^n
\langle n \rangle^\wedge_p \to S^n_p \to K(\Z,n)$.
By Theorem~~\ref{thm:fakes} and the previous section we know that
\begin{equation*}
\overline  G_{[n]} (S^n_p) = [K(\Z,n),B\aut{}{S^n \langle n
\rangle^\wedge_p}]/\{\pm 1\}, \qquad \overline G_{[n]}(S^n) =
[K(\Z,n),\prod_p B\aut{}{S^n \langle n \rangle^\wedge_p}]/\{\pm 1\}
\end{equation*}
where $-1$ acts on the integers by changing the sign.
The following corollary should certainly be
compared to Zabrodsky's \cite[Corollary~C']{MR901174}, where he
deals with locally finite homotopy groups in the source.

\begin{lemma}
For any $m \geq 3$ the space $B\aut{}{(S^n)^\wedge_p}$ is $K(\Z,
m)$-local. In particular it is $K(\Z, n)$-local.
\end{lemma}

\begin{proof}
Consider the universal fibration $(S^n)^\wedge_p \rightarrow
B\aut{*}{(S^n)^\wedge_p} \rightarrow B\aut{}{(S^n)^\wedge_p}$.
Since $(S^n)^\wedge_p$ is $B\mathbb Z/p$-local, it is also $K(\mathbb Z_{p^\infty}, m-1)$-local,
i.e. $\mapp{}{K(\mathbb Z_{p^\infty}, m-1)}{(S^n)^\wedge_p}$ is contractible. Likewise
$\mapp{}{K(\mathbb Z[1/p], m-1)}{(S^n)^\wedge_p} \simeq *$. Thus $(S^n)^\wedge_p$ is $K(\mathbb Z, m)$-local.
Next, we know that $B\aut{*}{(S^n)^\wedge_p}$ is $K(\Z, m)$-local by Proposition~\ref{prop:Klocal}. Hence, by
Zabrodsky's Lemma \cite[Lemma 4]{chuck96} we infer
that $\mapp{}{K(\Z, m)}{B\aut{}{(S^n)^\wedge_p}}$ is contractible.
\end{proof}

We compare now the classifying spaces $B\aut{}{(S^n)^\wedge_p}$ and
$B\aut{}{S^n \langle n \rangle^\wedge_p}$.

\begin{proposition}
\label{prop:padicdiscrete}
The space $\mapp{}{K(\Z, n)}{B\aut{}{S^n \langle n
\rangle^\wedge_p}}$ is homotopically discrete with $\Z^\wedge_p$
components.
\end{proposition}

\begin{proof}
Since there is a weak equivalence $B\aut{*}{S^n \langle n
\rangle^\wedge_p} \simeq B\aut{*}{(S^n)^\wedge_p}$ by
Corollary~\ref{cor:joe}.\eqref{cor:joe2}, we have a fibration
$K(\Z^\wedge_p, n) \rightarrow B\aut{}{S^n \langle n
\rangle^\wedge_p} \rightarrow B\aut{}{(S^n)^\wedge_p}$. The
proposition follows now from the previous lemma.
\end{proof}

We finally arrive at the unpointed mapping space.

\begin{theorem}
\label{thm:plocal}
The set of components of the space $\map{}{K(\Z, n)}{B\aut{}{S^n
\langle n \rangle^\wedge_p}}$ is $\Z^\wedge_p/(\Z^\wedge_p)^\times$.
It is in particular an infinite, countable set. Explicitely, $\overline
G_{[n]}(S^n_p)$ is in bijection with the set $\N_+$ of natural numbers with 
a disjoint base point $\ast$.
\end{theorem}

\begin{proof}
The homotopy fiber of the evaluation $\map{}{K(\Z, n)}{B\aut{}{S^n
\langle n \rangle^\wedge_p}} \rightarrow B\aut{}{S^n \langle n
\rangle^\wedge_p}$ is homotopically discrete and identifies with
$\Hom{\Z}{\Z^\wedge_p}$ by the last proposition. Moreover the
fundamental group $\pi_1 B\aut{}{S^n \langle n \rangle^\wedge_p}
\cong \pi_0 \aut{}{S^n \langle n \rangle^\wedge_p}$ coincides with
$\pi_1 B\aut{*}{(S^n)^\wedge_p} \cong (\Z^\wedge_p)^\times$, the
$p$-adic units. Their action on the $p$-adic integers comes from the
natural action on $\pi_n (S^n)^\wedge_p$. Thus the components of the
mapping space we are looking at is the quotient
$\Z^\wedge_p/(\Z^\wedge_p)^\times$.

Let $\N=\{0,1,2,\ldots\}$ be the set of natural numbers and
$\N_+$ the union of $\N$ with a disjoint base point $\ast$.  The
quotient $\Z^\wedge_p/\Z_p^\times$ is in bijection with the set $\N_+$
because any non-zero $p$-adic integer can be uniquely written as
$p^ku$ where $k \in \N$ and $u$ is a unit \cite{MR0344216}. The extended
genus set has been identified as the quotient of this set under the action of $\pm 1$.
However, since $-1$ in $\mathbb Z$ is sent in the $p$-adics to a unit, there are no
further identifications.
\end{proof}

\begin{construction}
\label{exmp:plocal}
Here is an explicit way to construct the countable set $\overline
G_L(S^n_p)$ of spaces $Y$ with $\pi_n Y \cong \Z$ and $Y\langle n
\rangle \simeq S^n \langle n \rangle^\wedge_p$.

The fibration $S^n \langle n \rangle^\wedge_p \to S^n_p \to K(\Z,n)$ is classified by a map 
$c \colon K(\Z,n) \to B\aut{}{S^n \langle n \rangle^\wedge_p}$. The proof of
Theorem~\ref{thm:plocal} shows that there is a bijection
   \begin{equation*}
     \N_+ \to
   [K(\Z,n),B\aut{}{S^n \langle n \rangle^\wedge_p}]/\{ \pm 1 \}
   \end{equation*}
   taking $\ast$ to the constant map and the nonnegative integer $k
   \in \N$ to $c \circ p^k$.

   Define the space $Y_{p,*}$ to be $S^n \langle n \rangle^\wedge_p \times K(\Z, n)$ and
   $Y_{p,k}$, $k \in \N$, to be the homotopy pull-back of $K(\Z,n)
   \xrightarrow{p^k} K(\Z,n) \leftarrow S^n_p$, or, equivalently, the
   homotopy fibre of $S^n_p \to K(\Z,n) \rightarrow
   K(\Z/p^k,n)$, where the second map is reduction mod $p^k$. The bijection is then given by
\[
   \begin{array}{lcl}
     \N_+ & \longrightarrow & \overline G_{[n]}(S^n_p) \\
     k & \longmapsto & Y_{p,k}
   \end{array}
\]
\end{construction}

We show now that one can detect which fake partially $p$-completed sphere one is considering by a simple cohomological computation.
We will be more precise in the proof.

\begin{proposition}
\label{prop:distinguish}
Ordinary cohomology distinguishes all elements in the extended Postnikov genus set $\overline G_{[n]}(S^n_p)$.
\end{proposition}

\begin{proof}
By Construction~\ref{exmp:plocal}
these fake partially completed spheres fit into a tower of fibrations
 \begin{equation*}
    \dots
 \rightarrow Y_{p, k+1} \xrightarrow{f} Y_{p, k} \rightarrow \dots
 \rightarrow Y_{p, 1} \xrightarrow{f} Y_{p, 0} = S^n_p
 \end{equation*}
with fibers $K(\Z/p, n-1)$. Let $\iota_k$ denote a generator of
$H^n(Y_{p, k}; \Z) \cong \Z$, chosen in such a way that the image of
$\iota_k$ under $f^*$ is $p \iota_{k+1}$.

At the prime $2$ the algebra structure is sufficient to distinguish
the fakes. Notice that the mod $2$ reduction of $\iota_1$ is
a polynomial generator detected in the mod $2$ cohomology of
$K(\Z/2, n-1)$, but $\iota_0$ is an exterior generator. Therefore
$(2^k \iota_k)^2 = 0$, but $(2^{k-1} \iota_k)^2 \neq 0$. This shows
that if $Y$ is any $(n-1)$-connected space with $\pi_n Y \cong \Z$ and $Y\langle n
\rangle \simeq (S^n \langle n \rangle)^\wedge_2$, and $\iota$
denotes a generator of $H^n(Y; \Z)$, then $Y \simeq Y_{2, k}$
where $k$ is the smallest integer such that $(2^k \iota)^2 = 0$.

At an odd prime $p$, the mod $p$ reduction of $\iota_1$ is an
\emph{exterior} generator detected in the mod $p$ cohomology of
$K(\Z/2, n-1)$, but it has non-trivial integral Steenrod operations,
such as $\mathcal P^1 \beta$, acting on it. This operation is
represented by classes of order $p$ in $H^{n+2p-1}(K(\Z, n); \Z)$
corresponding to the pair $(\mathcal P^1 \iota_n, \beta \mathcal P^1
\iota_n)$ in mod $p$ cohomology. This shows that if $Y$ is any space
with $\pi_n Y \cong \Z$ and $Y\langle n \rangle \simeq (S^n \langle
n \rangle)^\wedge_p$, and $\iota$ denotes a generator of $H^n(Y;
\Z)$, then $Y \simeq Y_{p, k}$ where $k$ is the smallest integer
such that $\mathcal P^1 (p^k \iota) = 0$.
\end{proof}

\begin{theorem}
\label{thm:classification}
The extended Postnikov genus set $\overline G_{[n]}(S^n)$ of homotopy types
of spaces $Y$ such that $Y[n] \simeq K(\Z, n)$ and $Y\langle n
\rangle \simeq S^n \langle n \rangle$ is uncountable, in bijection
with $\prod_p \N_+$, where the product is taken over all primes.
\end{theorem}

\begin{proof}
We apply Theorem~\ref{thm:fakes} and the identification
$B\aut{}{S^n \langle n \rangle} \simeq \prod B\aut{}{S^n \langle n
\rangle^\wedge_p}$ obtained above. Theorem~\ref{thm:plocal} shows
that the set of unpointed homotopy classes $[K(\Z, n), B\aut{}{S^n
\langle n \rangle}]$ is uncountable, in bijection with $\prod_p
\N_+$ and it remains to identify the action of the finite group
$\Aut{}{K(\Z, n)} \cong \Z/2$. But we have seen that it is trivial 
at each prime  since $-1$ is a unit in the $p$-adic integers.
Hence the action is trivial.
\end{proof}

Elaborating a little bit on Construction~\ref{exmp:plocal}, one can
explicitly construct all these fake spheres.

\begin{construction}
\label{rem:allfakespheres}
We identify $[K(\Z,n),B\aut{}{S^n \langle n \rangle}]$ with $\prod_p \N_+$. 
An element in this set is
a sequence $K= (k_p)$ consisting either of a natural number or the base
point $*$ for each prime $p$. For each such sequence consider the
homotopy pull-back $Y_K$ of the diagram
\[
\prod_p Y_{p,k} \to \prod K(\Z,n) \stackrel{\Delta}{\longleftarrow} K(\mathbb Z, n)
\]
where the spaces $Y_{p, k}$ have been
constructed in Example~\ref{exmp:plocal} and the second arrow is given by the
diagonal inclusion. The homotopy fiber of the map $Y_K \rightarrow K(\mathbb Z, n)$ is the product
$\prod_p S^n \langle n \rangle^\wedge_p \simeq S^n \langle n
\rangle$. The restriction to $B\aut{}{S^n \langle n
\rangle^\wedge_p}$ yields $Y_{p, k}$ which is classified by $k_p$.
This describes all spaces in $\overline G_{[n]}(S^n)$.
\end{construction}

Thus we have a good handle on all these fake spheres $S^n$. What is
so special about the good old $S^n$ among them? The answer is in Theorem~\ref{thm:Postnikovgenus}.

\begin{proposition}
\label{prop:difference}
Let $Y$ be a space such that $Y[n] \simeq K(\Z, n)$ and $Y\langle n
\rangle \simeq S^n \langle n \rangle$. Then, if $Y$ is a finite
complex, $Y$ has the homotopy type of $S^n$. \hfill{\qed}
\end{proposition}

Finally we address the question of what happens when one changes the
$n$-th Postnikov section for a higher one. The result will basically
remain the same. An explicit computation would prove to be more
difficult, but the concrete example of fake spheres we have produced
serve equally well now.

\begin{proposition}
Let $Y$ be an element in the extended Postnikov genus $\overline G_{[n]}(S^n_p)$,
and $m \geq n$. For any large enough prime $p$ we have that $Y[m]
\simeq S^n_p[m]$ and $Y \langle m \rangle \simeq S^n_p \langle m
\rangle$.
\end{proposition}

\begin{proof}
Since $Y$ has been constructed so that $Y \langle n \rangle
\simeq S^n \langle n \rangle^\wedge_p$, the same is true for a higher
connected cover. The claim about the $m$-th Postnikov section
follows by choosing $p > \dfrac{m-n+3}{2}$ so $\pi_* S^n$ has no $p$-torsion in
degrees $< m$.
\end{proof}

This implies again that, for any $m$, there are uncountably many
homotopy types of spaces which look like odd spheres through the
eyes of the $m$-th Postnikov section and $m$-connected cover. We end
the section with a related computation of the extended Postnikov
genus of complex projective spaces.

\begin{theorem}
\label{thm:CPngenus}
The extended Postnikov genus set $\overline G_{[2n+1]}(\C P^n)$ is uncountable for any $n \geq 1$.
\end{theorem}

\begin{proof}
Let $C=\C P^n[2n+1]$ be the $(2n+1)$-st Postnikov
section of $\C P^n$. There are fibrations $K(\Z, 2n+1) \to C \to
K(\Z,2)$ and $S^{2n+1}\langle 2n+1 \rangle \to \C P^n \to C$ where
$S^{2n+1}\langle 2n+1 \rangle$ decomposes as $\prod_p S^{2n+1}\langle 2n+1
\rangle^\wedge_p$.  Obstruction theory shows that $[\C P^n, \C P^n]
\to [C,C]$ is bijective so that self-maps of $\C P^n$ and $C$ are
classified up to homotopy by their degrees in $H^2(-;\Z)$. In
particular, $\Aut{}C \cong \Aut{}{\C P^n} \cong \Z^\times = \{ \pm 1
\}$ has two elements.

According to Theorem~\ref{thm:fakes}, $\overline G_{[2n+1]}(\C P^n) =
[C,B\aut{}{S^{2n+1}\langle 2n+1 \rangle}]/\Z^\times$.

Let $c \colon C \to B\aut{}{S^{2n+1}\langle 2n+1 \rangle}$ be the
classifying map for the standard $\C P^n$. We have seen that
$B\aut{}{S^{2n+1}\langle 2n+1 \rangle}$ splits as a product $\prod_p
B\aut{}{S^{2n+1}\langle 2n+1 \rangle^\wedge_p}$, and so
the classifying map $c$ decomposes as a product $c=\prod_pc_p \circ
\Delta$, where $\Delta \colon C \to \prod_pC$ is the diagonal map
and $c_p \colon C \to B\aut{}{S^{2n+1}\langle 2n+1
\rangle^\wedge_p}$. For any sequence $m=(m_p)_p \in \prod_p\Z$ of
integers, let $P^n_m$ be the space classified by $\prod c_p \circ
\prod m_p \circ \Delta \colon C \to \prod B\aut{}{X_p}$. For example
$\C P^n$ and $C \times S^{2n+1} \langle 2n+1 \rangle$ correspond
respectively to the constant sequences $(1)$ and $(0)$.

Consider now the restriction of one of the components $c_p \circ
m_p$ to the fiber $K(\Z, 2n+1)$. The degree $m_p$ map on $\C P^n$
induces the degree $m_p^n$ map on the cover $S^{2n+1}$, so that this
restriction corresponds to the class of $m_p^n$ in the coset
$\Z^\wedge_p/(\Z^\wedge_p)^\times$ we obtained in
Theorem~\ref{thm:plocal}. In particular, for any choice $m_p = p^k$
this restriction represents a different homotopy class in $[K(\Z,
2n+1), B\aut{}{S^{2n+1}\langle 2n+1 \rangle^\wedge_p}]$. In fact for
any choice $k_p \in \N$, the sequences $m = (p^{k_p})$ yield an
uncountable number of homotopy types of fake complex projective
spaces. Indeed the spaces $P_m^n$ are all distinct since the
homotopy pull-back of $P^n_m \rightarrow C \leftarrow K(\Z, 2n+1)$
is homotopy equivalent to the fake sphere described by the sequence
$(nk_p)$ as in Construction~\ref{rem:allfakespheres}.
\end{proof}


\bibliographystyle{amsplain}
\providecommand{\bysame}{\leavevmode\hbox to3em{\hrulefill}\thinspace}
\providecommand{\MR}{\relax\ifhmode\unskip\space\fi MR }
\providecommand{\MRhref}[2]{%
  \href{http://www.ams.org/mathscinet-getitem?mr=#1}{#2}
}
\providecommand{\href}[2]{#2}


\end{document}